\documentclass[11pt]{amsart}
\usepackage{amsmath,amsthm,amsfonts,amssymb,latexsym,epsfig}

\newtheorem{theorem}{Theorem}[section]

\newtheorem{lemma}{Lemma}[section]
\newtheorem{prop}{Proposition}[section]

\numberwithin{equation}{section}

\theoremstyle{definition}

\theoremstyle{remark}

\begin{document}
\title{A generalization of Hilbert's inequality}
\author{Peng Gao}
\address{Department of Mathematics, School of Mathematics and System Sciences, Beijing University of Aeronautics and Astronautics, P. R. China}
\email{penggao@buaa.edu.cn}
%%\date{August 30, 2007}

%%-------------------------------------------------------------------
\begin{abstract}In a generalization of the classical Hilbert inequality by Hardy, Littlewood and P\'{o}lya, the best constant for an inequality is determined provided that the generating function for the corresponding matrix satisfies certain monotonicity condition. In this paper, we determine the best constant for a class of inequalities when the monotonicity condition is no longer satisfied.
\end{abstract}

\maketitle
%%-----------------------------------------------------------------------
\section{Introduction}
\label{sec1}
%%------------------------------------------------------------------------
   
   Suppose throughout that $p >1, \frac{1}{p}+\frac{1}{q}=1$.
  Let $l^p$ be the Banach space of all complex sequences ${\bf x}=(x_n)_{n \geq 1}$ with norm
\begin{equation*}
   \|{\bf x}\|_p: =(\sum_{n=1}^{\infty}|x_n|^p)^{1/p} < \infty.
\end{equation*}

     Let $C=(c_{n,k})$ be a matrix acting on the $l^p$ space, the $l^{p}$ operator norm of $C$ is
   defined as
\begin{equation*}
\label{02}
    \| C \|_{p,p}=\sup_{\|{\bf x}\|_p = 1}\Big | \Big |C \cdot {\bf x}\Big | \Big |_p.
\end{equation*}

   The well-known Hilbert's inequality \cite[Theorem 315]{HLP} asserts that for ${\bf x} \in l^p, {\bf y} \in l^q$:
\begin{equation}
\label{1.6}
   \big | \sum^{\infty}_{i,j=1}\frac {x_iy_j}{i+j} \big | \leq \frac {\pi}{\sin \pi/p} \| {\bf x}\|_p \| {\bf y}\|_q. 
\end{equation}

   Let $H=(1/(i+j))_{i\geq 1,j \geq 1}$, then \cite[Theorem 286]{HLP} implies that inequality \eqref{1.6} is equivalent to $\| H \|_{p,p} \leq \pi/\sin (\pi/p)$. In fact, it is shown in \cite[Theorem 317]{HLP} that the constant $\pi/\sin (\pi/p)$ is best possible, hence $\| H \|_{p,p} = \pi/\sin (\pi/p)$.

   A generalization of Hilbert's inequality is given in \cite[Theorem 318]{HLP}. For a matrix $K=(K(i,j))_{i\geq 1,j \geq 1}$ with $K(x,y)$ satisfying the following conditions:
 \begin{align} 
 \label{1.02}
    1. &  K(x,y) \text{ is a non-negative, homogeneous function of degree $-1$};  \\
    2. &  \int^{\infty}_{0}K(x,1)x^{-1/q}dx=\int^{\infty}_{0}K(1,y)y^{-1/p}dy=k. \nonumber
 \end{align}
   Then, with one more condition (in what follows, we shall refer to this assumption as the decreasing assumption) that $K(x,1)x^{-1/q}, K(1,y)y^{-1/p}$ are strictly decreasing functions of $x>0,y>0$ respectively, \cite[Theorem 318]{HLP} asserts that $\|K\|_{p,p} \leq k$. In fact, in this case $\|K\|_{p,p} = k$ (see the remark on \cite[p. 229]{HLP}). 

   The proof for \cite[Theorem 318]{HLP} given in \cite{HLP} uses Schur's test to reduce the estimation of $\| K \|_{p,p}$ to the estimation of certain series and the decreasing assumption is to ensure that the series are bounded above by the corresponding integrals. In view of this, one sees that the decreasing assumption need not be necessary when determining $\| K \|_{p,p}$. It is therefore natural to study $\| K \|_{p,p}$ for a matrix $K=(K(i,j))_{i\geq 1,j \geq 1}$ with $K(x,y)$ satisfying the conditions in \eqref{1.02} only.
   
     We now focus on a type of matrices of the form:  $H(\alpha, \beta)=(i^{\alpha}j^{\beta}/(i+j)^{\alpha+\beta+1})_{i\geq 1,j \geq 1}$. We note here that when either $\alpha>1/q$ or $\beta>1/p$, then $H(\alpha, \beta)$ satisfies the conditions in \eqref{1.02} but not the decreasing assumption.
     
    Some results on $\| H(\alpha, \beta)\|_{p,p}$ can be deduced from a result of Bennett. Recall (see \cite{A&S}) that the beta function $B(x,y), x>0, y>0$ is defined as 
\begin{align}
\label{1.3}
   B(x,y)=\int^1_0t^{x-1}(1-t)^{y-1}dt=\int^{\infty}_{0}\frac {t^{x-1}}{(1+t)^{x+y}}dt.
\end{align} 

    Bennett's result (see \cite[Proposition 2]{B} and the discussions that follow) can be stated as follows:
\begin{theorem}
\label{thm0} Let $p > 1$ be fixed. Let $\mu$ be a positive measure on $(0,1)$. Let $K=(k_{i,j})_{i \geq 1, j \geq 1}$ be given by
\begin{align*}
  k_{i,j}=\int^1_0\binom {i+j-2}{j-1}t^{i-1}(1-t)^{j}d \mu(t).
\end{align*}
  Then $\| K \|_{p,p} \leq \int^1_0t^{-1/q}(1-t)^{1/q}d \mu(t)$. 
\end{theorem}
  
  By taking $d \mu(t)=t^{1-\alpha}(1-t)^{-\beta}dt$ in the above theorem, we readily deduce the following
\begin{theorem}
\label{thm0'} Let $p > 1$ and $\alpha<1+1/p, \beta <1+1/q$ be fixed. Let $M(\alpha, \beta)$ be a matrix given by
\begin{align*}
  M(\alpha, \beta)_{i,j}=\binom {i+j-2}{j-1}B(i+1-\alpha, j+1-\beta), \quad i,j \geq 1.
\end{align*}
  Then $\| M(\alpha, \beta) \|_{p,p} \leq B(1+1/p-\alpha, 1+1/q-\beta)$. 
\end{theorem}
  
   Note that the case $\alpha=\beta=1$ in the above theorem corresponds to Hilbert's inequality (a stronger form with the corresponding matrix being $(1/(i+j-1))_{i \geq 1, j \geq 1}$). The case $\alpha=1, \beta=0$ in the above theorem corresponds to the matrix studied explicitly by Bennett in \cite[Proposition 2]{B}. Note that $H(0,1)_{i,j}\leq M(1,0)_{i,j}$ when $i, j \geq 1$. Similarly, it is easy to check that $H(1,1)_{i,j} \leq M(0,0)_{i,j}$ when $i, j \geq 1$.  It follows from Theorem \ref{thm0'} that $\| H(0,1) \|_{p, p}  \leq \| M(1,0) \|_{p,p} \leq B(1/p, 1+1/q), \| H(1,1) \|_{p,p} \leq \| M(0,0) \|_{p,p} \leq B(1+1/p, 1+1/q)$.  On the other hand,  Lemma \ref{lem40} in Section \ref{sec 4} implies that $\| H(0,1) \|_{p, p}  \geq B(1/p, 1+1/q), \| H(1,1) \|_{p,p} \geq B(1+1/p, 1+1/q)$.  We therefore deduce the following
 \begin{theorem}
\label{thm0''} Let $p > 1$ be fixed and let $H(0,1)=(j/(i+j)^2)_{i \geq 1, j \geq 1}$, $H(1,1)=(ij/(i+j)^3)_{i \geq 1, j \geq 1}$. Then $\| H(0,1) \|_{p,p}=\pi/(q\cdot \sin(\pi/p)), \| H(1,1) \|_{p,p}=\pi/(2p q\cdot \sin(\pi/p))$. 
\end{theorem}

  To establish results analogue to that given in Theorem \ref{thm0''}, one hopes to have that for $\alpha<1+1/p, \beta <1+1/q$, $i,j \geq 1$,
\begin{align*}
  H(1-\alpha, 1-\beta)_{i,j} \leq M(\alpha, \beta)_{i,j}, 
\end{align*}
   where $M(\alpha, \beta)$ is defined as in Theorem \ref{thm0'}. However, the above inequality does not always hold. For example, one checks directly that when $\alpha=\beta=-1, i=j$, the above inequality fails to hold when $i \rightarrow \infty$. Similarly, with the help of Stirling's approximation, one can show that the above inequality fails to hold when $\alpha=\beta=-1/2, i=j$, $i \rightarrow \infty$. 

   In this paper, we use the approach in the proof of \cite[Theorem 318]{HLP} to study $\| H(\alpha, \beta) \|_{p,p}$. It follows from \cite[Theorem 286]{HLP} and the approach on \cite[p. 229]{HLP} that for any matrix $K=(K(i,j))_{i \geq 1, j\geq 1}$ with $K(x,y)$ a non-negative, homogeneous function of degree $-1$, we have $\| K \|_{p,p} \leq k$ provided that
\begin{align*}
  \sum^{\infty}_{i=1}K(i,j)\left ( \frac {i}{j} \right )^{-1/q} \leq k, \quad \sum^{\infty}_{j=1}K(i,j)\left ( \frac {j}{i} \right )^{-1/p} \leq k.
\end{align*}

   Apply the above argument to $K(x,y)=x^{\alpha}y^{\beta}/(x+y)^{\alpha+\beta+1}$, we see that (with the help of Lemma \ref{lem40}) $\| H(\alpha, \beta) \|_{p,p}=B(\alpha+1/p, \beta+1/q)$ for any $\alpha>-1/p, \beta>-1/q$ provided that we can show for any $\lambda>-1, s>\lambda+1$:
\begin{align}
\label{1.4}
   \sum^{\infty}_{m=1}\frac {m^{\lambda}}{(m+n)^{s}} \leq B(\lambda+1, s-\lambda-1)n^{1+\lambda-s}.
\end{align}
    
   We note that the above inequality is valid when $\lambda<0$ by the integral test. In \cite[Lemma 2]{K&P}, it is shown that the above inequality is valid when $0<s \leq 2, -1<\lambda < s-1$ and $2<s \leq 14, -1<\lambda \leq 1$. 
In the next section, we extend this result to the case of $1< \lambda \leq 2, \lambda+1<s \leq 5$ to prove the following 
\begin{theorem}
\label{mainthm}
  Let $p >1$ be fixed and let $H(\alpha, \beta)=(i^{\alpha}j^{\beta}/(i+j)^{\alpha+\beta+1})_{i\geq 1,j \geq 1}$ , then $\| H(\alpha, \beta) \|_{p,p}=B(\alpha+1/p, \beta+1/q)$ when $ -1/p<\alpha \leq 2,   -1/q<\beta \leq 2$.
\end{theorem}

    We note here that \cite[Theorem 286]{HLP} again implies that Theorem \ref{mainthm} is equivalent to the statement that with the best possible constant $C(\alpha,\beta, p)=B(\alpha+1/p, \beta+1/q)$, for ${\bf x} \in l^q, {\bf y} \in l^p$:
\begin{align}
\label{1.2}
   \big | \sum^{\infty}_{i,j=1}\frac {i^{\alpha}j^{\beta}}{(i+j)^{\alpha+\beta+1}}x_iy_j \big | \leq C(\alpha, \beta, p)\| {\bf x}\|_q \| {\bf y}\|_p. 
\end{align}

   The proof for Theorem \ref{mainthm} relies on the estimation of the series given in \eqref{1.4}. When $-1 < \lambda \leq 0, s> \lambda+1$, this series is studied by Bennett and Jameson in \cite[Proposition 14]{B&J}. In the same paper, they raised an open question (see the remark below \cite[Proposition 13]{B&J}) on whether the following sequence is increasing with $n$ when $\alpha >1$:
\begin{align}
\label{1}
  \frac 1{n^{\alpha+2}}\sum^{n-1}_{r=1}r^\alpha(n-r).
\end{align}

   In this paper, we give an affirmative answer to the above open question of Bennett and Jameson. We prove in Section \ref{sec 3} the following:
\begin{theorem}
\label{thm1} Let $\alpha >1$ be fixed, the sequence defined in \eqref{1} is increasing with $n$.
\end{theorem}
%%----------------------------------------------------------------------------
\section{Proof of Theorem \ref{mainthm}}
\label{sec 4} \setcounter{equation}{0}
%%----------------------------------------------------------------------------
   Note that when $\alpha \leq 2, \beta \leq 2$, then $\alpha+\beta+1\leq 5$. Thus, by our discussion in Section \ref{sec1}, we see that Theorem \ref{mainthm} follows from a combination of \cite[Lemma 2]{K&P} and the following two lemmas:
\begin{lemma}
\label{lem40}
   Let $p>1, \alpha > -1/p,  \beta >-1/q$ be fixed, let $H(\alpha, \beta)=(i^{\alpha}j^{\beta}/(i+j)^{\alpha+\beta+1})_{i \geq 1, j \geq 1}$, then $\| H(\alpha, \beta) \|_{p,p} \geq B(\alpha+1/p,\beta+1/q)$.
\end{lemma}
\begin{proof}
   By \cite[Theorem 286]{HLP}, it suffices to show that if there exists a constant $C(\alpha, \beta, p)$ such that inequality \eqref{1.2} holds for any ${\bf x} \in l^q, {\bf y} \in l^p$, then $C(\alpha, \beta, p) \geq B(\alpha+1/p, \beta+1/q)$. We follow the method given on \cite[p. 233]{HLP} by setting $x_i=i^{-1/q}, y_j=j^{-1/p}$ when $1 \leq i, j \leq N$ and $x_i=y_j=0$ otherwise to see that
\begin{align*}
   \sum^{\infty}_{i,j=1}\frac {i^{\alpha}j^{\beta}}{(i+j)^{\alpha+\beta+1}}x_iy_j=\sum^{N}_{i=1}i^{\alpha-1/q}\sum^{N}_{j=1}\frac {j^{\beta-1/p}}{(i+j)^{\alpha+\beta+1}}, \quad \| {\bf x}\|_q \| {\bf y}\|_p=\sum^{N}_{i=1}\frac 1{i}.
\end{align*}
   For any given $\epsilon>0$, we choose $N$ large enough such that (note that our assumption on $\alpha$ and $\beta$  ensures that the infinite series in the following expression converges)
\begin{align*}
   \sum^{N}_{j=1}\frac {j^{\beta-1/p}}{(i+j)^{\alpha+\beta+1}}>(1-\epsilon)\sum^{\infty}_{j=1}\frac {j^{\beta-1/p}}{(i+j)^{\alpha+\beta+1}}.
\end{align*}
  Note that the function $x^{\beta-1/p}/(x+i)^{\alpha+\beta+1}$ is increasing when $x < (\beta-1/p)i/(\alpha+1+1/p)$ and decreasing when $x>(\beta-1/p)i/(\alpha+1+1/p)$. It follows that 
\begin{align*}
 & \frac {j^{\beta-1/p}}{(i+j)^{\alpha+\beta+1}}\geq \int^{j}_{j-1}\frac {x^{\beta-1/p}}{(i+x)^{\alpha+\beta+1}}dx, \quad j \leq (\beta-1/p)i/(\alpha+1+1/p),\\
 & \frac {j^{\beta-1/p}}{(i+j)^{\alpha+\beta+1}} \geq \int^{j+1}_{j}\frac {x^{\beta-1/p}}{(i+x)^{\alpha+\beta+1}}dx, \quad j > (\beta-1/p)i/(\alpha+1+1/p).
\end{align*}
  We then deduce easily that
\begin{align*}
   \sum^{\infty}_{j=1}\frac {j^{\beta-1/p}}{(i+j)^{\alpha+\beta+1}}\geq \int^{\infty}_{0}\frac {x^{\beta-1/p}}{(i+x)^{\alpha+\beta+1}}dx-\frac {C}{i^{\alpha+1+1/p}}=B(\alpha+\frac 1{p}, \beta+\frac 1{q})i^{-(\alpha+1/p)}-\frac {C}{i^{\alpha+1+1/p}},
\end{align*}
  where $C$ is a constant depending on $\alpha, \beta, p$ and we used \eqref{1.3} to evaluate the integration above.

  As $\sum^{\infty}_{i=1}i^{-2}<\infty$, we deduce that
\begin{align*}
&  \sum^{N}_{i=1}i^{\alpha-1/q}\sum^{N}_{j=1}\frac {j^{\beta-1/p}}{(i+j)^{\alpha+\beta+1}} \geq B(\alpha+\frac 1{p}, \beta+\frac 1{q})(1-\epsilon)\sum^{N}_{i=1}\frac 1i+C' \\
& =B(\alpha+\frac 1{p}, \beta+\frac 1{q})(1-\epsilon)\| {\bf x}\|_q \| {\bf y}\|_p+(1-\epsilon)C',
\end{align*}
  where $C'$ is a constant depending on $\alpha, \beta, p$. By letting $N \rightarrow \infty$, we see that the constant $C(\alpha, \beta, p)$ in \eqref{1.2} satisfies $C(\alpha, \beta, p) \geq B(\alpha+1/p, \beta+1/q)$ and this completes the proof.
\end{proof}
%%-------------------------------------------------------------------
\begin{lemma}
\label{lem1}
   Let $1< \lambda \leq 2$, then inequality \eqref{1.4} holds for $\lambda+1<s \leq 5$.
\end{lemma}
\begin{proof}
   We apply the following Euler-Maclaurin summation formula (see \cite[p. 152]{K&P}), which asserts that for $f(x) \in C^2[1, \infty)$ such that $\sum^{\infty}_{k=1} f(k)<\infty, \int^{\infty}_1f(t)dt<\infty$ and $\lim_{x \rightarrow \infty}f^{(r)}(x) = 0, 0 \leq r \leq 1$, then the following equality holds:
\begin{align}
\label{2.01}
  \sum^{\infty}_{k=1} f(k)=\int^{\infty}_1f(t)dt+\frac {f(1)}{2}-\frac 1{12}f'(1)-\frac 1{2}\int^{\infty}_{1}B_2(\{ t \})f''(t)dt,
\end{align}
  where we denote by $\{x \}$ the fractional part of $x$,
the unique real number in $[0, 1)$ such that $x-\{x\} \in \mathbb{Z}$ and $B_2(x)=x^2-x+1/2$ the second Bernoulli polynomial.  

   Further, it follows from \cite[Proposition 9.2.3]{Cohen} that if $f \in C^6[1,\infty), \lim_{x \rightarrow \infty}f^{(r)}(x) = 0, 0 \leq r \leq 5$, 
and $f^{(r)}(x) < 0, r=4,6$, then the following inequality holds:
\begin{align*}
   \frac 1{720}f'(1)-\frac 1{720\cdot 42}f^{(3)}(1) \leq -\frac 1{2}\int^{\infty}_{1}B_2(\{ t \})f(t)dt \leq  \frac 1{720}f'(1).
\end{align*}

   We use the notion on \cite[p. 152]{K&P} to define for real numbers $\lambda, s$ and $t>0$,
\begin{align*}
   f_{s, \lambda, n}(t)=\frac {t^{\lambda}}{(t+n)^s}.
\end{align*}
   We note that
\begin{align*}
   f'_{s, \lambda, n}(t) &=\frac {nst^{\lambda-1}}{(t+n)^{s+1}}-\frac {(s-\lambda)t^{\lambda-1}}{(t+n)^s}, \\
   f''_{s, \lambda, n}(t) &=g_{s, \lambda, n}(t)+h_{s, \lambda, n}(t),
\end{align*}
   where
\begin{align*}
  g_{s, \lambda, n}(t)=\frac {(s+1-\lambda)(s-\lambda)t^{\lambda-2}}{(t+n)^s}+\frac {n^2s(s+1)t^{\lambda-2}}{(t+n)^{s+2}}, \quad h_{s, \lambda, n}(t)=-\frac {2ns(s+1-\lambda)t^{\lambda-2}}{(t+n)^{s+1}}.
\end{align*}
   It follows from our discussion above that when $1<\lambda \leq 2, s>\lambda$,
\begin{align*}
  -\frac 1{2}\int^{\infty}_{1}B_2(\{ t \})g_{s, \lambda, n}(t)dt & \leq \frac 1{720}g'_{s, \lambda, n}(1)-\frac 1{720\cdot 42}g^{(3)}_{s, \lambda, n}(1),  \\ -\frac 1{2}\int^{\infty}_{1}B_2(\{ t \})h_{s, \lambda, n}(t)dt & \leq \frac 1{720}h'_{s, \lambda, n}(1).
\end{align*}
   It now follows from \eqref{2.01} that
\begin{align*}
 & \sum^{\infty}_{k=1}  f_{s, \lambda, n}(k)-\int^{\infty}_0f_{s, \lambda, n}(t)dt \\
\leq & -\int^1_0f_{s, \lambda, n}(t)dt+\frac {f_{s, \lambda, n}(1)}{2}-\frac {f'_{s, \lambda, n}(1)}{12}+\frac 1{720}f^{(3)}_{s, \lambda, n}(1)-\frac 1{720\cdot 42}g^{(3)}_{s, \lambda, n}(1).
\end{align*}
   Thus, inequality \eqref{1.4} is valid provided that we show
\begin{align}
\label{2.02}
   D_{s,\lambda}(n):=\int^1_0f_{s, \lambda, n}(t)dt-\frac 1{2}f_{s, \lambda, n}(1)+\frac {f'_{s, \lambda, n}(1)}{12}-\frac 1{720}f^{(3)}_{s, \lambda, n}(1)+\frac 1{720\cdot 42}g^{(3)}_{s, \lambda, n}(1)\geq 0.
\end{align}
  Integration by parts shows that (see \cite[p. 153]{K&P}) 
\begin{align*}
   \int^1_0f_{s, \lambda, n}(t)dt \geq \sum^{5}_{i=0}\frac 1{(n+1)^{s+i}}\cdot \frac {\prod^{i}_{j=1}(s+j-1)}{\prod^{i+1}_{j=1}(j+\lambda)},
\end{align*}
  where we define the empty product to be $1$.

  We also have
\begin{align*}
 -\frac 1{2}f_{s, \lambda, n}(1) &=-\frac 1{2(n+1)^s}, \\
 \frac {f'_{s, \lambda, n}(1)}{12} &=\frac {\lambda}{12(n+1)^s}-\frac {s}{12(n+1)^{s+1}}, \\
 -\frac 1{720}f^{(3)}_{s, \lambda, n}(1) &=-\frac 1{720}\left ( \frac {(\lambda-1)(\lambda-2)(\lambda-3)}{(n+1)^s}-\frac {3s\lambda(\lambda-1)}{(n+1)^{s+1}}+\frac {3s(s+1)\lambda}{(n+1)^{s+2}}-\frac {s(s+1)(s+2)}{(n+1)^{s+3}} \right ),\\
g^{(3)}_{s, \lambda, n}(1) &=\frac {(s+1-\lambda)(s-\lambda)(\lambda-2)(\lambda-3)(\lambda-4)}{(n+1)^s}-\frac {3s(s+1-\lambda)(s-\lambda)(\lambda-2)(\lambda-3)}{(n+1)^{s+1}}\\
 &+\frac {3s(s+1)(s+1-\lambda)(s-\lambda)(\lambda-2)}{(n+1)^{s+2}}-\frac {3s(s+1)(s+2)(s+1-\lambda)(s-\lambda)}{(n+1)^{s+3}} \\
& +\frac {n^2s(s+1)(\lambda-2)(\lambda-3)(\lambda-4)}{(n+1)^{s+2}}-\frac {3n^2s(s+1)(s+2)(\lambda-2)(\lambda-3)}{(n+1)^{s+3}} \\
& +\frac {3n^2s(s+1)(s+2)(s+3)(\lambda-2)}{(n+1)^{s+4}}-\frac {n^2s(s+1)(s+2)(s+3)(s+4)}{(n+1)^{s+5}}.
\end{align*}
   Apply this in \eqref{2.02}, we see that
\begin{align*}
   D_{s,\lambda}(n) \geq \sum^{5}_{i=0}\frac 1{(n+1)^{s+i}}D_i(s,\lambda),
\end{align*}
   where
\begin{align*}
   D_0(s,\lambda) &=\frac 1{1+\lambda}-\frac 1{2}+\frac {\lambda}{12}-\frac {(\lambda-1)(\lambda-2)(\lambda-3)}{720} \\
& +\frac {(\lambda-2)(\lambda-3)(\lambda-4)}{720\cdot 42}\left ( (s+1-\lambda)(s-\lambda)+s(s+1)\right ), \\
  D_1(s,\lambda) &=\frac s{(1+\lambda)(2+\lambda)}-\frac s{12}+\frac {3s\lambda(\lambda-1)}{720} \\
& -\frac {s(\lambda-2)(\lambda-3)}{720\cdot 42}\left ( 3(s+1-\lambda)(s-\lambda)+2(s+1)(\lambda-4)+3(s+1)(s+2)\right ), \\ 
  D_2(s,\lambda) & =\frac {s(s+1)}{(1+\lambda)(2+\lambda)(3+\lambda)}-\frac {3s(s+1)\lambda}{720}+\frac {s(s+1)(\lambda-2)}{720\cdot 42} \cdot \\
& \cdot \Big (3(s+1-\lambda)(s-\lambda)+(\lambda-3)(\lambda-4)+6(s+2)(\lambda-3)+3(s+2)(s+3) \Big ), \\
 D_3(s,\lambda) & =\frac {s(s+1)(s+2)}{(1+\lambda)(2+\lambda)(3+\lambda)(4+\lambda)}+\frac {s(s+1)(s+2)}{720}  -\frac {s(s+1)(s+2)}{720\cdot 42} \cdot \\
& \cdot \Big (3(s+1-\lambda)(s-\lambda)+3(\lambda-2)(\lambda-3)+6(s+3)(\lambda-2)+(s+3)(s+4)\Big  ), \\
 D_4(s,\lambda) & =\frac {s(s+1)(s+2)(s+3)}{(1+\lambda)(2+\lambda)(3+\lambda)(4+\lambda)(5+\lambda)}+\frac {s(s+1)(s+2)(s+3)}{720\cdot 42}\left (3(\lambda-2)+2(s+4)\right ), \\
D_5(s,\lambda) & =s(s+1)(s+2)(s+3)(s+4)\left (\frac {1}{(1+\lambda)(2+\lambda)(3+\lambda)(4+\lambda)(5+\lambda)(6+\lambda)}-\frac {1}{720\cdot 42}\right ).
\end{align*}

  We now assume that $s > \lambda+1, 1< \lambda \leq 2$,  it is then easy to see that $D_4(s,\lambda) \geq 0, D_5(s,\lambda) \geq 0$ when $1<\lambda \leq 2$. We apply the bounds $(1+\lambda)(2+\lambda)(3+\lambda) \leq 60, (1+\lambda)(2+\lambda)(3+\lambda)(4+\lambda) \leq 360$ when $1< \lambda \leq 2$ to see that $D_2(s,\lambda) \geq 0, D_3(s,\lambda) \geq 0$ respectively, when
\begin{align*}
 6\cdot 42 & \geq 3(s+1-\lambda)(s-\lambda)+(\lambda-3)(\lambda-4)+6(s+2)(\lambda-3)+3(s+2)(s+3), \\
 126 & \geq 3(s+1-\lambda)(s-\lambda)+3(\lambda-2)(\lambda-3)+6(s+3)(\lambda-2)+(s+3)(s+4).
\end{align*}
   For fixed $s$, the right-hand expressions above are increasing functions of $\lambda$ and hence are maximized when $\lambda=2$. Thus, the above inequalities are valid for all $1< \lambda \leq 2, s>\lambda+1$ provided that the following inequalities are valid: 
\begin{align*}
  6\cdot 42 & \geq 3(s-1)(s-2)+2-6(s+2)+3(s+2)(s+3)=6s^2+14, \\
  126 & \geq 3(s-1)(s-2)+(s+3)(s+4)=4s^2-2s+18.
\end{align*}
   One checks readily that the above inequalities are valid for $\lambda+1< s \leq 5$.

   Note that
\begin{align*}
  \frac s{(1+\lambda)(2+\lambda)}-\frac s{12}=\frac {s(2-\lambda)(5+\lambda)}{12(1+\lambda)(2+\lambda)}.
\end{align*}
   It is easy to see that $D_1(s,\lambda) \geq 0$ when $\lambda=2$ for any $s>\lambda+1$. When $1< \lambda <2$, we see that $D_1(s,\lambda) \geq 0$ is equivalent to the following inequality:
\begin{align*}
 & \frac {60\cdot 42(5+\lambda)}{(1+\lambda)(2+\lambda)}+\frac {3\cdot 42\lambda(\lambda-1)}{(2-\lambda)} \\
 \geq & (3-\lambda)\left ( 3(s+1-\lambda)(s-\lambda)+2(s+1)(\lambda-4)+3(s+1)(s+2)\right ).
\end{align*}
   It is easy to see that the left-hand side expression above is $\geq 60\cdot 42\cdot 6/(3 \cdot 4)=30 \cdot 42$ while the right-hand side expression above is $\leq 3( 3(s+1-1)(s-1)+3(s+1)(s+2))=18(s^2+s+1)$. It follows that the above inequality is valid for $\lambda+1 < s \leq 5$.

   Note that
\begin{align*}
   \frac 1{1+\lambda}-\frac 1{2}+\frac {\lambda}{12}=\frac {(\lambda-2)(\lambda-3)}{12(1+\lambda)}.
\end{align*}
   Thus, when $1< \lambda \leq 2$ to see that $D_0(s,\lambda) \geq 0$ follows from 
\begin{align*}
  \frac {60\cdot 42}{1+\lambda}-42\cdot (\lambda-1) \geq (4-\lambda)\left ( (s+1-\lambda)(s-\lambda)+s(s+1)\right ).
\end{align*}
   It is easy to see that the left-hand side expression above is $\geq 19\cdot 42$. We apply the bound $4-\lambda \leq 3,  (s+1-\lambda)(s-\lambda)\leq s(s-1)$ to see that $D_0(s,\lambda) \geq 0$ follows from 
\begin{align*}
  19\cdot 42 \geq 6s^2,
\end{align*}
   which is valid for $\lambda+1 < s \leq 5$. This completes the proof. 
\end{proof}
  
%%----------------------------------------------------------------------------
\section{Further Discussions}
\label{sec 3} \setcounter{equation}{0}
%%----------------------------------------------------------------------------
%%---------------------------------------------------------------------------
  In this section, we first improve the result of \cite[Lemma 2]{K&P} by showing that inequality \eqref{1.4} is valid for any $s>2$ when $\lambda=1$.
\begin{prop}
\label{lem31}
   Let $s >2, C_{s}=((s-2)(s-1))^{-1}$, then for 
  for any integer $n \geq 1$,
\begin{align}
\label{2.1}
   \sum^{\infty}_{m=1}\frac {m}{(n+m)^{s}} \leq \frac {C_{s}}{n^{s-2}}.
\end{align}
\end{prop}
\begin{proof}
  Note first that the condition $s>2$ ensures that the infinite series in \eqref{2.1} converges. Let
\begin{align*}
  f_{s}(n):= \frac {C_{s}}{n^{s-2}}-\left( \sum^{\infty}_{m=1}\frac {1}{(n+m)^{s-1}}-n\sum^{\infty}_{m=1}\frac {1}{(n+m)^{s}} \right ).
\end{align*}

    As $\sum^{\infty}_{i=1}i^{1-s}<\infty$, it follows that 
\begin{align*}
  \lim_{n \rightarrow \infty}\sum^{\infty}_{m=1}\frac {1}{(n+m)^{s-1}}=0.
\end{align*}

    Note also that
 \begin{align*}
  0 \leq \lim_{n \rightarrow \infty}n\sum^{\infty}_{m=1}\frac {1}{(n+m)^{s}} \leq \lim_{n \rightarrow \infty}n\int^{\infty}_{0}\frac {1}{(n+x)^{s}}dx=\lim_{n \rightarrow \infty}\frac  {1}{(s-1)n^{s-2}}=0.
\end{align*}
   We then deduce that $\lim_{n \rightarrow \infty}f_{s}(n)=0$. Therefore, it suffices to show that $f_{s}(n)-f_{s}(n+1)\geq 0$. Calculation shows that
\begin{align*}
  f_{s}(n)-f_{s}(n+1)=\frac {C_{s}}{n^{s-2}}-\frac {C_{s}}{(n+1)^{s-2}}-\sum^{\infty}_{m=1}\frac {1}{(n+m)^{s}}.
\end{align*} 
  Note that \cite[Lemma 3]{C&L} asserts that for $s>1$,
\begin{align*}
 \sum^{\infty}_{i=k}\frac 1{i^{s}} \leq \frac {s}{s-1} \cdot \frac 1{k^{s}-(k-1)^{s}}.
\end{align*} 
  We apply this to see that it suffices to show that
\begin{align*}
 \frac {C_{s}}{n^{s-2}}-\frac {C_{s}}{(n+1)^{s-2}}- \frac {s}{s-1} \cdot \frac 1{(n+1)^{s}-n^{s}} \geq 0.
\end{align*} 
   We can recast the above inequality as
\begin{align*}
%%\label{31}
 \left (\int^{n+1}_nx^{s-1}dx\right ) \left ( \int^{n+1}_n x^{1-s}dx\right )\geq 1.
\end{align*} 
   As the above inequality follows from the Cauchy-Schwarz inequality, this completes the proof.
\end{proof}
%%-----------------------------------------------------------------------------------------
  
   We end this paper by proving Theorem \ref{thm1}. It amounts to show that
\begin{align*}
    \frac 1{n^{\alpha+2}}\sum^{n}_{r=1}r^{\alpha}(n-r) \leq  \frac 1{(n+1)^{\alpha+2}}\sum^{n}_{r=1}r^{\alpha}(n+1-r).
\end{align*}
   The above inequality can be rewritten as
\begin{align*}
    \left( n-\frac {n^{\alpha+2}}{(n+1)^{\alpha+2}-n^{\alpha+2}}\right )\sum^{n}_{r=1}r^{\alpha} \leq \sum^{n}_{r=1}r^{\alpha+1}.
\end{align*}
   The above inequality holds trivially for $n=1$ and therefore by induction, it suffices to show that
\begin{align*}
&  \left( n+1-\frac {(n+1)^{\alpha+2}}{(n+2)^{\alpha+2}-(n+1)^{\alpha+2}}\right )\sum^{n+1}_{r=1}r^{\alpha} -\left( n-\frac {n^{\alpha+2}}{(n+1)^{\alpha+2}-n^{\alpha+2}}\right )\sum^{n}_{r=1}r^{\alpha} \\
& \leq \sum^{n+1}_{r=1}r^{\alpha+1}-\sum^{n}_{r=1}r^{\alpha+1}=(n+1)^{\alpha+1}.
\end{align*}
   After simplification, the above inequality becomes
\begin{align*}
  \left( \frac{((n+1)/n)^{\alpha+2}}{((n+1)/n)^{\alpha+2}-1}-\frac {1}{((n+2)/(n+1))^{\alpha+2}-1}\right )\sum^{n}_{r=1}r^{\alpha} \leq \frac {(n+1)^{\alpha}}{((n+2)/(n+1))^{\alpha+2}-1}.
\end{align*}
   Further simplification yields
\begin{align*}
  \left( \frac{(n+2)^{\alpha+2}-(n+1)^{\alpha+2}}{(n+1)^{\alpha+2}-n^{\alpha+2}}-1\right )\sum^{n}_{r=1}r^{\alpha} \leq (n+1)^{\alpha}.
\end{align*}
   The above inequality is equivalent to the following inequality:
\begin{align}
\label{10}
  \frac {\sum^{n}_{r=1}r^{\alpha}}{\sum^{n+1}_{r=1}r^{\alpha}}\leq \frac {(n+1)^{\alpha+2}-n^{\alpha+2}}{(n+2)^{\alpha+2}-(n+1)^{\alpha+2}}.
\end{align}
  
   We note the following lemma:
\begin{lemma}[{\cite[Lemma 2.1]{Xu}}]
\label{lem3}
   Let $\{B_n \}^{\infty}_{n=1}$ and $\{C_n \}^{\infty}_{n=1}$ be strictly increasing positive sequences with
   $B_1/B_2 \leq C_1 / C_2$. If for any integer $n \geq 1$,
\begin{equation*}
  \frac {B_{n+1}-B_n}{B_{n+2}-B_{n+1}} \leq  \frac
  {C_{n+1}-C_n}{C_{n+2}-C_{n+1}}.
\end{equation*}
  Then $B_{n}/B_{n+1} \leq C_{n} / C_{n+1}$ for any integer $n \geq 1$.
\end{lemma}

   Applying the above lemma with $B_n=\sum^{n}_{r=1}r^{\alpha}, C_n=(n+1)^{\alpha+2}-n^{\alpha+2}$ and observing that the sequence $\{C_n \}^{\infty}_{n=1}$ is strictly increasing by the Mean Value Theorem, we see that inequality \eqref{10} holds provided that we have
\begin{align*}
 &  \frac {1}{1^{\alpha}+2^{\alpha}}\leq \frac {2^{\alpha+2}-1}{3^{\alpha+2}-2^{\alpha+2}}, \\
 &  \frac {(n+1)^{\alpha}}{(n+2)^{\alpha}} \leq \frac {(n+2)^{\alpha+2}-2(n+1)^{\alpha+2}+n^{\alpha+2}}{(n+3)^{\alpha+2}-2(n+2)^{\alpha+2}+(n+1)^{\alpha+2}}, \quad n \geq 1.
\end{align*}
 
   It is easy to see that the first inequality above follows from the case $n=0$ of the second inequality. It therefore remains to prove the second inequality above for $n \geq 0$. We recast the second inequality above as $f(1/(n+2)) \leq f(1/(n+1))$, where
\begin{align*}
  f(x)=x^{-2}\Big ( (1+x)^{\alpha+2}+(1-x)^{\alpha+2}-2\Big ).
\end{align*}
   Note that
  \begin{align*}
  \frac {x^2}{2}\cdot f'(x)=\frac {g(x)+g(0)}{2}-\frac 1{x}\int^x_0g(t)dt, 
  \end{align*}
      where
  \begin{align*}
      g(x)=(\alpha+2)(1+x)^{\alpha+1}-(\alpha+2)(1-x)^{\alpha+1}.
  \end{align*}
    As $g(x)$ is convex when $0<x \leq 1$,  it follows from the Hermite-Hadamard inequality \cite[p. 10]{MPF} that $f'(x)\geq 0$ when $0<x \leq 1$ and this completes the proof.

%%----------------------------------------------------------------------------------

%%------------------------------------------------------------------------


\begin{thebibliography}{99}
%%----------------------------------------------------------------------------------
%%---------------------------------------------------------------------------
\bibitem{A&S} M. Abramowitz and I. A. Stegun, eds., {\em Handbook of mathematical functions with formulas, graphs and mathematical tables}, Dover, New York, 1966.
%%-------------------------------------------------------------------------------------------------------------------
\bibitem{B} G. Bennett, Lower bounds for matrices, {\em Linear Algebra Appl.}, {\bf 82} (1986),  81--98.
%%----------------------------------------------------------------------------------------
\bibitem{B&J} G. Bennett and G. Jameson, Monotonic averages of convex functions,
{\em J. Math. Anal. Appl.}, {\bf 252} (2000), 410--430.
%%---------------------------------------------------------------------------
\bibitem{C&L} J.A. Cochran and C.S. Lee, Inequalities related to Hardy's and Heinig's, {\em Math. Proc. Cambridge
Philos. Soc.}, {\bf 96} (1984), 1--7.
%%-----------------------------------------------------------------------------------------
\bibitem{Cohen} H. Cohen, {\em
Number Theory-Volume II: Analytic and Modern Tools}, Graduate Texts in
Mathematics Vol. 240, Springer, New York 2007.
%%-----------------------------------------------------------------------------------------
\bibitem{HLP} G. H. Hardy, J. E. Littlewood and G. P\'{o}lya, {\em
Inequalities}, Cambridge Univ. Press, 1934.
%%----------------------------------------------------------------------------------------
\bibitem{K&P} M. Krni\' c and J. Pe\v cari\' c, Extension of Hilbert's inequality,
{\em J. Math. Anal. Appl.}, {\bf 324} (2006), 150--160.
%%--------------------------------------------------------------------------------------------
\bibitem{MPF} D. S. Mitrinovi\' c,  J. E. Pe\v cari\'c and A. M. Fink, {\em Classical and new inequalities in analysis}, Kluwer Academic Publishers Group, Dordrecht, 1993.
%%-----------------------------------------------------------------------------------------
\bibitem{Xu}
Z. K. Xu and D. P. Xu, A general form of Alzer's inequality, {\em
Comput. Math. Appl.}, {\bf 44} (2002), 365-373.
\end{thebibliography}
\end{document}